\documentclass[12pt]{amsart}

\usepackage{enumerate, amsmath, amsthm, amsfonts, amssymb, xy,  mathrsfs, graphicx, paralist, fancyvrb}
\usepackage[usenames, dvipsnames]{xcolor}
\usepackage[margin=1in]{geometry} 
\usepackage[bookmarks, colorlinks=true, linkcolor=blue, citecolor=blue, urlcolor=blue]{hyperref}

\input xy
\xyoption{all}

\numberwithin{equation}{section}
\newtheorem{theorem}[equation]{Theorem}

\newtheorem{proposition}[equation]{Proposition}

\newtheorem{corollary}[equation]{Corollary}
\newtheorem{conjecture}[equation]{Conjecture}
\newtheorem{problem}[equation]{Problem}

\theoremstyle{definition}
\newtheorem{rmk}[equation]{Remark}
\newenvironment{remark}[1][]{\begin{rmk}[#1] \pushQED{\qed}}{\popQED \end{rmk}}
\newtheorem{eg}[equation]{Example}
\newenvironment{example}[1][]{\begin{eg}[#1] \pushQED{\qed}}{\popQED \end{eg}}
\newtheorem{defn}[equation]{Definition}

\makeatletter
\@addtoreset{equation}{section}
\renewcommand{\thesubsection}{%
  \ifnum\c@subsection<1 \@arabic\c@section
  \else \thesection.\@arabic\c@subsection
  \fi
}
\makeatother

\newcommand{\bA}{\mathbf{A}}

\newcommand{\bC}{\mathbf{C}}

\newcommand{\rD}{\mathrm{D}}

\newcommand{\rH}{\mathrm{H}}

\newcommand{\bM}{\mathbf{M}}

\newcommand{\bP}{\mathbf{P}}

\newcommand{\bS}{\mathbf{S}}

\newcommand{\fS}{\mathfrak{S}}

\newcommand{\bZ}{\mathbf{Z}}

\newcommand{\bk}{\mathbf{k}}

\renewcommand{\phi}{\varphi}

\newcommand{\ul}[1]{\underline{#1}}

\newcommand{\arxiv}[1]{\href{http://arxiv.org/abs/#1}{{\tt arXiv:#1}}}

\makeatletter
\def\Ddots{\mathinner{\mkern1mu\raise\p@
\vbox{\kern7\p@\hbox{.}}\mkern2mu
\raise4\p@\hbox{.}\mkern2mu\raise7\p@\hbox{.}\mkern1mu}}
\makeatother

\DeclareMathOperator{\reg}{reg}

\DeclareMathOperator{\Sym}{Sym}

\DeclareMathOperator{\Tor}{Tor}
\DeclareMathOperator{\pdim}{pdim}

\newcommand{\GL}{\mathbf{GL}}

\usepackage[alphabetic,lite,nobysame]{amsrefs}
\usepackage{youngtab}

\title{On some modules supported in the Chow variety}
\author{Claudiu Raicu}
\address{Department of Mathematics, University of Notre Dame, 255 Hurley, Notre Dame, IN 46556 USA\newline
\indent Institute of Mathematics ``Simion Stoilow'' of the Romanian Academy}
\email{craicu@nd.edu}
\thanks{CR was supported by NSF DMS-1901886.}

\author{Steven V Sam}
\address{Department of Mathematics, University of California, San Diego, CA USA}
\email{ssam@ucsd.edu}
\thanks{SS was supported by NSF DMS-1849173.}

\author{Jerzy Weyman}
\address{Department of Mathematics, Jagiellonian University, Krak\'ow, Poland}
\email{jerzy.weyman@gmail.com}
\thanks{JW was supported by MAESTRO NCN-UMO-2019/34/A/ST1/00263
and NAWA POWROTy -PPN/PPo/2018/1/00013/U/00001 grants, as well as by NSF DMS 1802067}

\date{August 24, 2021}
\dedicatory{Dedicated to Bernd Sturmfels on the occasion of his 60th birthday.}

\newcommand{\des}{\mathrm{des}}
\newcommand{\maj}{\mathrm{maj}}

\def\lra{\longrightarrow}

\newcommand{\defi}[1]{{\bf \upshape\sffamily #1}}
\newcommand{\op}[1]{\operatorname{#1}}
\newcommand{\mc}[1]{\mathcal{#1}}
\newcommand{\mf}[1]{\mathfrak{#1}}

\renewcommand{\a}{\alpha}
\newcommand{\chr}{\operatorname{char}}
\newcommand{\Ext}{\operatorname{Ext}}
\newcommand{\Hom}{\operatorname{Hom}}
\newcommand{\onto}{\twoheadrightarrow}
\newcommand{\oo}{\otimes}

\begin{document}

\maketitle

\begin{abstract} The study of Chow varieties of decomposable forms lies at the confluence of algebraic geometry, commutative algebra, representation theory and combinatorics. There are many open questions about homological properties of Chow varieties and interesting classes of modules supported on them. The goal of this note is to survey some fundamental constructions and properties of these objects, and to propose some new directions of research. Our main focus will be on the study of certain maximal Cohen--Macaulay modules of covariants supported on Chow varieties, and on defining equations and syzygies. We also explain how to assemble $\Tor$ groups over Veronese subalgebras into modules over a Chow variety, leading to a result on the polynomial growth of these groups.
\end{abstract}

\section{Introduction}

In this article we collect several results on graded modules supported in the Chow variety~$Y_{d,n}$, defined as the set of homogeneous polynomials of degree $d$ in $n+1$ variables which decompose into a product of linear factors. Our interest is in the study of certain Cohen--Macaulay modules of covariants from a homological and representation theoretic perspective, analyzing the shape of their minimal resolution and their equivariant structure. We also include a discussion of Brill equations, rank varieties, and the study of $Y_{d,n}$ for small values of the parameters, as well as a general polynomiality statement for $\Tor$ groups of Veronese subalgebras. Our hope is to give a flavor of the questions and results surrounding Chow varieties, that will spark future research in this area.

Let $\bk$ be an algebraically closed field, and let $U\simeq\bk^{n+1}$ be an $(n+1)$-dimensional vector space. Thinking of $U^*$ as the set of linear forms in $n+1$ variables, we can identify the space $\Sym^d(U^*)$ with the affine variety $X_{d,n}$ of homogeneous polynomials of degree $d$ in $n+1$ variables. We consider the subvariety of \defi{completely decomposable forms} in $X_{d,n}$, defined as
\begin{equation}\label{eq:def-Ydn}
 Y_{d,n} = \{ F \in X_{d,n} : F = \ell_1\cdots\ell_d,\text{ with }\ell_i\in U^*\}.
\end{equation}
An equivalent description of $Y_{d,n}$ is as the affine cone over the \defi{Chow variety} of $0$-dimensional cycles of length $d$ in $\bP^n$. The standard parametrization of $Y_{d,n}$ (using the tuples $(\ell_1,\dots,\ell_d)$) realizes the coordinate ring $\bk[Y_{d,n}]$ as a subalgebra of $A_{d,n}$, the homogeneous coordinate ring of the Segre product $\left(\bP^n\right)^{\times d}$. There is a natural action of the symmetric group $\mf{S}_d$ on $A_{d,n}$, and the invariant subring $B_{d,n}=A_{d,n}^{\mf{S}_d}$ gives the normalization of $\bk[Y_{d,n}]$. Moreover, the normalization map has an interesting connection to Foulkes' Conjecture about plethysm coefficients in algebraic combinatorics (see Section~\ref{s:char-free}). 

More generally, for every partition $\lambda$ of $d$ we can construct the module of covariants
\[
  M_{\lambda}= \Hom_{\mf{S}_d}(V_{\lambda},A_{d,n}),
\]
where $V_\lambda$ is the corresponding irreducible $\fS_d$-representation.

In characteristic zero, each $M_{\lambda}$ can be thought of as a maximal Cohen--Macaulay module supported on $Y_{d,n}$. In Section~\ref{sec:mods-Ydn} we discuss the action of the duality functor on these modules, as well as some consequences regarding the shape of their minimal resolution (over the coordinate ring of $X_{d,n}$). We analyze the case $d=2$, which corresponds to modules supported on rank $\leq 2$ matrices, as well as the case $n=1$. Notice that $Y_{d,1}$ is an affine space, so the modules $M_{\lambda}$ are free in this case, but the description of their generators involves interesting combinatorics related to the well-studied statistics of \defi{descents} and \defi{major indices}, and to Kostka--Foulkes polynomials. One way to interpret the formulas in the case $n=1$ is as generalizations of Hermite Reciprocity, which states that, for all $a,b \ge 0$, we have an isomorphism $\Sym^a(\Sym^b U) \cong \Sym^b(\Sym^a U)$ as $\GL(U)$-representations when $\dim(U)=2$. For general parameters $d,n$, we take the first step in the study of the syzygies of $M_{\lambda}$ by establishing a bound on their Castelnuovo--Mumford regularity.

In Section~\ref{s:char-zero} we analyze the Brill equations and recall computational results about their relation to defining ideals of $Y_{d,n}$.
In particular, we recall why the ideals generated by Brill equations are not radical for $n\ge d\ge 3$. This is due to certain ``rank equations''
of the same degree but different representation type that also vanish on $X_{d,n}$ for $d\ge n\ge 3$. We also treat in detail the normalization of the Chow variety $X_{3,2}$ 
and point out some connections to classical results on ternary cubics going back to Aronhold.
 
Finally, in Section~\ref{sec:Hermite} we prove a result about the finite generation of the $\Tor$ modules in the spirit of stability in representation theory, and describe an algebraic version of a polynomiality result for $\Tor$ groups by Yang.

\section{The variety of decomposable forms} \label{s:char-free}

In this section we assume that $\bk$ is an algebraically closed field of arbitrary characteristic. We recall some of the basic properties of $Y_{d,n}$, discussing its normalization map, along with some connections to Foulkes' conjecture. We begin with a natural parametrization of $Y_{d,n}$, obtained by considering the variety
\begin{subequations}
  \begin{equation}\label{eq:def-Zdn}
  Z_{d,n}=\lbrace (\ell_1,\ell_2,\ldots ,\ell_d)\ |\ \ell_1 ,\ldots ,\ell_d\in U^*\rbrace
 \end{equation}
of $d$-tuples of linear forms, along with the multiplication map
\begin{equation}\label{eq:param-Ydn}
  \mu_{d,n} \colon Z_{d,n}\lra Y_{d,n}\subseteq X_{d,n}, \qquad (\ell_1,\ldots ,\ell_d) \mapsto \ell_1\cdots \ell_d.
\end{equation}
If we consider the associated pull-back map
\begin{equation}\label{eq:def-mu-sharp}
 \mu_{d,n}^{\sharp}\colon \bk[X_{d,n}] \lra \bk[Z_{d,n}],
\end{equation}
then the coordinate ring of $Y_{d,n}$ can be identified as
\[ \bk[Y_{d,n}] = \op{Im}\left(\mu_{d,n}^{\sharp}\right).\]
If we let $J_{d,n}$ denote the defining ideal of $Y_{d,n}$ then we have
\begin{equation}\label{eq:def-Jdn}
 J_{d,n} = \ker\left(\mu_{d,n}^{\sharp}\right).
\end{equation}
\end{subequations}
The group $\GL(U)$ of invertible transformations of $U$ acts on all the varieties described so far, making their coordinate rings into representations of $\GL(U)$. We have for instance that 
\[\bk[X_{d,n}] = \Sym(\rD^d U) = \bigoplus_{m\geq 0} \Sym^m(\rD^d U),\]
where $\rD^d U = \Sym^d(U^*)^*$ is the \defi{$d$-th divided power} of $U$, and 
\[\bk[Z_{d,n}] = \Sym(U \oplus\cdots \oplus U) = \Sym(U)^{\oo d} = \bigoplus_{a_1,\dots ,a_d\geq 0} \Sym^{a_1}U \otimes \cdots \otimes \Sym^{a_d} U.\]
Notice that the identification above provides $\bk[Z_{d,n}]$ with a natural $\bZ^d$-grading, which can be seen as coming from the action of the torus $(\bk^*)^{\times d}$ by scaling on each tensor factor. There is also an action of the symmetric group $\mf{S}_d$ on $\Sym(U)^{\oo d}$ by permutations of the tensor factors, and the morphism $\mu_{d,n}^{\sharp}$ is induced by the natural inclusion
\begin{equation}\label{eq:phi-sharp-gens}
 \rD^d(U) \hookrightarrow \bk[Z_{d,n}]_{(1,\dots,1)}=U^{\oo d}
\end{equation}
as the subspace of symmetric tensors. From now on we write
\[ R_{d,n} = \bk[X_{d,n}],\text{ and }A_{d,n} = \bigoplus_{m\geq 0} \bk[Z_{d,n}]_{(m,\dots,m)} = \bigoplus_{m\geq 0}\Sym^m U \oo \cdots \oo \Sym^m U,\]
noting that $A_{d,n} = R_{d,n}^T$ is a ring of invariants for the action of the torus
\[
  T = \{(t_1,\dots,t_d)\in(\bk^*)^{\times d} : t_1\cdots t_d=1\},
\]
and in particular that $A_{d,n}$ is normal. Since the image of $\mu_{d,n}^{\sharp}$ is generated by $\mf{S}_d$-invariant elements, it follows that
\[ \bk[Y_{d,n}] \subseteq A_{d,n}^{\mf{S}_d} = \bigoplus_{m\geq 0} \rD^d\left(\Sym^m U\right) =: B_{d,n},\]
where $B_{d,n}$ is also normal. We will prove the following (see also \cites{brion,neeman,nag}).

\begin{theorem}\label{thm:normalization-Ydn}
 The algebra $B_{d,n}$ is the normalization of $\bk[Y_{d,n}]$.
\end{theorem}

Before explaining the proof of Theorem~\ref{thm:normalization-Ydn}, we proceed with a more concrete description of the constructions above. We choose a basis $x_0,\dots,x_n$ for $U^*$, which in turn induces a monomial basis
\[
  \{x^{\a} = x_0^{\a_0}\cdots x_n^{\a_n} : \a_0+\cdots+\a_n = d\}
\]
for $\Sym^d(U^*)$.
If we let $\{z_{\a}\}_{\a}$ denote the dual basis of $\rD^d(U)$, then we can think of its elements as the coefficients of the generic $d$-form
\[
  \sum_{|\a|=d} z_{\a}\cdot x^{\a},
\]
and make the identification $R_{d,n} = \bk[z_{\a}]$. For $1\leq i\leq d$, we write $u^{(i)}_j$, $j=0,\dots,n$, for the coordinate functions on the affine space $U$ parametrizing the forms $\ell_i$ in \eqref{eq:def-Zdn}. We get an identification
\[ \bk[Z_{d,n}] = \bk[u^{(i)}_j],\]
and the morphism $\mu_{d,n}^{\sharp}$ sends $z_{\a}$ to the coefficient of $x^{\a}$ in the expansion of
\begin{equation}\label{eq:generic-prod-lin-forms}
\prod_{i=1}^d\left(\sum_{j=0}^n u^{(i)}_j\cdot x_j\right).
\end{equation}

\begin{example}\label{ex:d=2-n=2}
 Suppose that $d=2$ and $n=2$, so \eqref{eq:generic-prod-lin-forms} becomes
 \[ \left(u^{(1)}_0\cdot x_0+u^{(1)}_1\cdot x_1+u^{(1)}_2\cdot x_2\right)\cdot \left(u^{(2)}_0\cdot x_0+u^{(2)}_1\cdot x_1+u^{(2)}_2\cdot x_2\right).\]
The morphism $\mu_{2,2}^{\sharp}$ sends
\[z_{(2,0,0)}\mapsto u^{(1)}_0\cdot u^{(2)}_0,\quad z_{(0,2,0)}\mapsto u^{(1)}_1\cdot u^{(2)}_1,\quad z_{(0,0,2)}\mapsto u^{(1)}_2\cdot u^{(2)}_2,\ \]
\[z_{(1,1,0)}\mapsto u^{(1)}_0\cdot u^{(2)}_1 + u^{(1)}_1\cdot u^{(2)}_0,\quad z_{(1,0,1)}\mapsto u^{(1)}_0\cdot u^{(2)}_2 + u^{(1)}_2\cdot u^{(2)}_0,\quad z_{(0,1,1)}\mapsto u^{(1)}_1\cdot u^{(2)}_2 + u^{(1)}_2\cdot u^{(2)}_1.\] 
If $\chr(\bk)\neq 2$ then $Y_{2,2}$ parametrizes quadrics of rank $\leq 2$. With the usual identification between quadratic forms and symmetric matrices, we get that $J_{2,2}$ is a principal ideal, generated by the determinant of
\[
Z = \begin{bmatrix}
2z_{(2,0,0)} & z_{(1,1,0)} & z_{(1,0,1)} \\
z_{(1,1,0)} & 2z_{(0,2,0)} & z_{(0,1,1)} \\
z_{(1,0,1)} & z_{(0,1,1)} & 2z_{(0,0,2)}
\end{bmatrix}.
\]
If $\chr(\bk)=2$ then $J_{2,2}$ is still principal, generated by
\[z_{(2,0,0)}z_{(0,1,1)}^2 + z_{(0,2,0)}z_{(1,0,1)}^2 + z_{(0,0,2)}z_{(1,1,0)}^2 + z_{(0,1,1)}z_{(1,0,1)}z_{(1,1,0)}.\]
In both cases, $Y_{2,2}$ is a hypersurface of degree $3$ in $X_{2,2}\simeq \bA^6$. The singular locus $Y_{2,2}^{\rm sing}$ has codimension two in $Y_{2,2}$ (given by the quadrics of rank $\leq 1$ if $\chr(\bk)\neq 2$, and by the $3$-plane $z_{(0,1,1)}=z_{(1,0,1)}=z_{(1,1,0)}=0$ if $\chr(\bk)=2$), hence $Y_{2,2}$ is normal and $B_{2,2}=\bk[Y_{2,2}]$.
\end{example}

It will be useful to denote
\[
  u_{\ul{b}} = \prod_{i=1}^d u_{b_i}^{(i)},\text{ for }\ul{b}=(b_1,\cdots,b_d)\in \{0,\dots,n\}^{\times d},
\]
so the expansion of \eqref{eq:generic-prod-lin-forms} is given by
\[
  \sum_{b_1,\cdots,b_d=0}^n u_{(b_1,\dots,b_d)} x_{b_1}\cdots x_{b_d}.
\]
We consider the $\mf{S}_d$-action on the cartesian product $\{0,\dots,n\}^{\times d}$, write $O_{\ul{b}}$ for the orbit of $\ul{b}$, and write $\ul{a}\sim\ul{b}$ if $\ul{a}\in O_{\ul{b}}$. Note that every orbit $O_{\ul{b}}$ has a unique representative $\ul{a}\sim\ul{b}$ with $n\geq a_1\geq \cdots \geq a_d\geq 0$. If we let
\begin{equation}\label{eq:def-Ua}
 U_{\ul{a}} = \sum_{\ul{b}\sim\ul{a}} u_{\ul{b}}
\end{equation}
then we get that $\bk[Y_{d,n}]$ is the $\bk$-algebra generated by the forms $U_{\ul{a}}$.

\begin{proposition}\label{prop:finite-over-Ydn}
 The algebras $A_{d,n}$ and $B_{d,n}$ are finitely generated $\bk[Y_{d,n}]$-modules.
\end{proposition}

\begin{proof}
 Since $B_{d,n}\subseteq A_{d,n}$, it is enough to prove the result for $A_{d,n}$. Using the graded Nakayama lemma, it suffices to prove that the forms $U_{\ul{a}}$ generate an ideal of finite colength in $A_{d,n}$. Since $A_{d,n}$ is the coordinate ring of a $d$-fold Segre product $\op{Seg}=(\bP^n)^{\times d}$, this is further equivalent to checking that (after passing to an algebraic closure of $\bk$) the forms in \eqref{eq:def-Ua} have no common zeros on $\op{Seg}$. This follows from \cite[Proposition~5.2]{ES-Boij}, where it is shown that in fact the forms
 \[ \sum_{|\ul{a}|=\ell} U_{\ul{a}},\text{ with }0\leq\ell\leq dn,\]
have no common zeroes on $\op{Seg}$.
\end{proof}

\begin{proof}[Proof of Theorem~\ref{thm:normalization-Ydn}]
 Since $B_{d,n}$ is normal and finite over $\bk[Y_{d,n}]$ by Proposition~\ref{prop:finite-over-Ydn}, it suffices to show that the inclusion is birational, that is, the function field $F=\bk(Y_{d,n})$ is equal to the fraction field $\op{Frac}(B_{d,n})$. We let $K=\op{Frac}(A_{d,n})$ and observe that $\mf{S}_d$ acts faithfully on $K$, hence $\op{Frac}(B_{d,n}) = K^{\mf{S}_d}$ satisfies
 \[ \left[ K : \op{Frac}(B_{d,n}) \right] = d!.\]
 Since $F\subseteq \op{Frac}(B_{d,n})$, in order to prove equality it suffices to check that $\left[ K:F \right] \leq d!$, which we do next. If we let 
\[t=\prod_{i=1}^d u^{(i)}_0,\text{ and }v^{(i)}_j = \frac{u^{(i)}_j}{u^{(i)}_0}\text{ for }1\leq i\leq d,\ 1\leq j\leq n,\]
then $K = \bk(t,v^{(i)}_j)$ is a field of rational functions in $dn+1$ independent variables. We note that $t=u_{\ul{0}}= U_{\ul{0}} \in \bk[Y_{d,n}]$, and let $V_{\ul{a}}=U_{\ul{a}}/U_{\ul{0}}$, so that $F= \bk(t,V_{\ul{a}})$. After dividing \eqref{eq:generic-prod-lin-forms} by $t$, we get that $V_{\ul{a}}$ are the coefficients in the expansion
\begin{equation}\label{eq:dehomog-prod-lin-forms}
\prod_{i=1}^d\left(x_0 + \sum_{j=1}^n v^{(i)}_j\cdot x_j\right).
\end{equation}
Observe that
\[
  V_{(1^k,0^{d-k})} = e_k(v^{(1)}_1,\dots,v^{(d)}_1),\ 0\leq k\leq d,
\]
are the elementary symmetric polynomials in $v^{(1)}_1,\dots,v^{(d)}_1$, so for $n=1$ we have $F=K^{\mf{S}_d}$ by the fundamental theorem of symmetric polynomials. If $n\geq 2$ then we consider the intermediate extension $F\subseteq E \subseteq K$ defined by
\[
  E = F\left(v_1^{(1)},\dots,v_1^{(d)}\right),
\]
and note that since $F$ contains all $e_k(v^{(1)}_1,\dots,v^{(d)}_1)$, we must have $[E:F] \leq d!$. To finish the proof, we will show that $E=K$, or equivalently, that each $v^{(i)}_j\in E$ when $j\geq 2$. We consider the elements of $E$ defined by
\[
  e^{(i)}_k = e_k(v^{(1)}_1,\dots,\widehat{v^{(i)}_1},\dots,v^{(d)}_1) \text{ for }1\leq i\leq d,\ 0\leq k\leq d-1,
\]
where $e_k$ denotes an elementary symmetric polynomial, and $\widehat{\bullet}$ denotes a missing term. We consider the matrix $M$ whose entry in row $k+1$ and column $i$ is $e^{(i)}_k$. If we fix $j\geq 2$ and consider the coefficients of $x_j x_0^{d-1},x_j x_0^{d-2}x_1,\dots,x_j x_1^{d-1}$ in \eqref{eq:dehomog-prod-lin-forms}, then we obtain an identity
\[
  M \cdot \begin{bmatrix} v^{(1)}_j \\ v^{(2)}_j \\ \vdots \\ v^{(d)}_j \end{bmatrix} = \begin{bmatrix} V_{(j,0^{d-1})} \\ V_{(j,1,0^{d-2})} \\ \vdots \\ V_{(j,1^{d-1})} \end{bmatrix} \in F^d \subseteq E^d.
\]
Since $M$ has entries in $E$, in order to prove that $v^{(i)}_j\in E$ it suffices to show that $\det(M)\neq 0$. This follows from the fact that $\det(M)$ is a homogeneous polynomial in $v^{(i)}_j$ of degree $d\choose 2$, and the coefficient of the term
\[ \left(v^{(1)}_1\right)^{d-1} \left(v^{(2)}_1\right)^{d-2}\cdots \left(v^{(d-1)}_1\right)\]
is equal to one (this term arises in a unique way from expanding the product of the entries on the main diagonal of $M$). 
\end{proof}

\begin{remark}\label{rem:normality}
 The polynomials $V_{\ul{a}}$ in the proof of Theorem~\ref{thm:normalization-Ydn} are the \defi{elementary multi-symmetric polynomials} in the $d$ sets of variables $v^{(i)}_\bullet$, $i=1,\dots,d$. In characteristic zero, they generate the algebra of multi-symmetric polynomials $\bk[v^{(i)}_j]^{\mf{S}_d}$ (see \cite{schlafli}, \cite{noether}, \cite[Section~II,~Chapter~I]{macmahon}, \cite[Chapter~II,~Section~A.3]{weyl}), but this is no longer the case in positive characteristic \cites{nag,neeman,dalbec,briand,rydh}. This was used to show that the projective Chow variety $\op{Proj}(\bk[Y_{d,n}])$ is normal in characteristic zero, and that it usually fails to be normal in positive characteristic \cites{brion,neeman}.
\end{remark}

We can now write down an exact sequence
\[
  0\lra J_{d,n} \lra R_{d,n} \xrightarrow{\phi_{d,n}}  B_{d,n}\lra C_{d,n}\to 0,
\]
where $C_{d,n} = B_{d,n}/\bk[Y_{d,n}]$, and $\phi_{d,n}$ is the map $\mu^{\sharp}_{d,n}$ from \eqref{eq:def-mu-sharp}. Remark~\ref{rem:normality} implies that $C_{d,n}$ has finite length in characteristic zero, but in general not in positive characteristic, and it would be interesting to investigate the module $C_{d,n}$ further. The degree $m$ component of the map $\phi_{d,n}$ is
\[\phi_{d,n,m} \colon \Sym^m(\rD^d U)\to \rD^d(\Sym^m U),\]
and is usually referred to as the \defi{Foulkes--Howe map}, or the \defi{Hermite--Hadamard--Howe map} (see also \cites{howe} and the survey \cite{Landsberg}). The condition that $C_{d,n}$ has finite length is equivalent to the surjectivity of $\phi_{d,n,m}$ for $m\gg 0$. It remains an open problem to determine $m_0(d,n)$, the smallest value for which $\phi_{d,n,m}$ is surjective when $m\geq m_0(d,n)$, although an effective bound in characteristic zero was found by Brion \cite{brion-french}. Another question that we don't know the answer to is whether $B_{d,n}$ is Cohen--Macaulay in positive characteristic (in characteristic zero, $B_{d,n}$ is a ring of invariants for the reductive group $\fS_d \ltimes T$, hence it has in fact rational singularities by a theorem of Boutot \cite{boutot}). We will return to discussing more characteristic zero examples regarding the defining equations of $Y_{d,n}$ in Section~\ref{s:char-zero}. 

One important motivation for the study of the maps $\phi_{d,n,m}$ comes from the following.

\begin{conjecture}[\cite{foulkes}]\label{conj:foulkes}
  If $\chr(\bk)=0$ then for any $d\ge m$ the $\GL(U)$-representation $\Sym^m(\Sym^d U)$ is isomorphic to a subrepresentation of $\Sym^d (\Sym^m U)$.
\end{conjecture}

Indeed, in \cite[Section~2]{howe} suggests that one could try to prove Conjecture~\ref{conj:foulkes} by showing that $\phi_{d,n,m}$ is injective when $d\geq m$, and surjective when $d\leq m$. Equivalently, this would mean that $m_0(d,n)=d$. When $n=1$ this statement is true, and it is a manifestation of Hermite reciprocity. The approach fails however for a general $n$, as M\"uller and Neunh\"offer \cite{Muller} compute that the map $\phi_{5,5,5}$ is not an isomorphism. Despite the setback, Conjecture~\ref{conj:foulkes} remains open in general, and there is strong evidence to support it:
\begin{itemize}
\item as explained in \cite{cheung}, the conjecture is true when $m\leq 5$;
\item \cite{Muller} verifies computationally the cases $m\le 4$, $d\le 14$, and $m+d\le 17$, $d\le 12$;
\item \cite{brion} proves that the conjecture is true asymptotically (for $d\gg m$).
\end{itemize}

\section{Cohen--Macaulay modules supported on $Y_{d,n}$}\label{sec:mods-Ydn}

Throughout this section we assume that $\bk$ is algebraically closed with $\chr(\bk)=0$, and our goal is to study modules of covariants for the $\mf{S}_d$-action on $A_{d,n}$. In this setting, divided powers are isomorphic to symmetric powers, so we will ignore the distinction. We write $A=A_{d,n}$, $B=B_{d,n}$, and recall that $B=A^{\mf{S}_d}$. 

We also recall from \cite[Section~4.2]{ful-har} that the irreducible $\mf{S}_d$-representations are indexed by partitions $\lambda$ of size $|\lambda|=d$. We write $V_{\lambda}$ for the irreducible corresponding to $\lambda$, with the convention that $V_{(d)}$ is the trivial representation, while $V_{(1^d)}$ is the sign representation. We write $\bS_{\lambda}$ for the \defi{Schur functor} associated to $\lambda$ (see \cite[\S 2]{weyman} for basics on Schur functors, noting that $\bS_\lambda$ is denoted there by $L_{\lambda^\dagger}$, where $\dagger$ denotes the transpose partition), and we have $\bS_{(d)}U = \Sym^d U$, $\bS_{(1^d)}=\bigwedge^d U$. By \defi{Schur-Weyl duality} \cite[Exercise~6.30]{ful-har} we have a decomposition
\begin{subequations}
  \begin{equation}\label{eq:SW-decomp-A}
 A = \bigoplus_{|\lambda|=d} M_{\lambda} \oo V_{\lambda},
\end{equation}
where
\begin{equation}\label{eq:GL-decomp-M-lam}
 M_{\lambda} = \Hom_{\mf{S}_d}(V_{\lambda},A) = \bigoplus_{m\geq 0}\bS_{\lambda}(\Sym^m U)
\end{equation}
\end{subequations}
is the \defi{module of $\lambda$-covariants} for the $\mf{S}_d$-action on $A$. Note that $B=M_{(d)}$, and since $A$ is Cohen--Macaulay, finite and torsion-free over $B$, each of the direct summands $M_{\lambda}$ of $A$ is a maximal Cohen--Macaulay (MCM) $B$-module.

We write $R=R_{d,n}$, and using the natural maps $R\onto\bk[Y_{d,n}]\subseteq B_{d,n}$, we can think of each $M_{\lambda}$ as a Cohen--Macaulay $R$-module supported on $Y_{d,n}$. The projective dimension of each $M_{\lambda}$ (as an $R$-module) is then given by the codimension of $Y_{d,n}$ in $X_{d,n}$:
\begin{equation}\label{eq:pdim-M-lam}
\pdim_R(M_{\lambda}) = {d+n \choose n} - dn-1.
\end{equation}
We propose the following problem.

\begin{problem}\label{prob:res-M-lam}
 Describe the minimal free resolution of $M_{\lambda}$ as an $R$-module.
\end{problem}

Observe that since $\mf{S}_d$ acts trivially on $R$, Problem~\ref{prob:res-M-lam} is equivalent to understanding the $\mf{S}_d$-equivariant resolution $F_{\bullet}$ of $A$ as an $R$-module. We then get that $\Hom_{\mf{S}_d}\left(V_{\lambda},F_{\bullet}\right)$ is the minimal resolution of $M_{\lambda}$, and in particular
\[ \Tor_i^R(M_{\lambda},\bk) = \Hom_{\mf{S}_d}\left(V_{\lambda},\Tor_i^R(A,\bk)\right)\text{ for all }i.\]

The algebra $A$ is not only Cohen--Macaulay, but also Gorenstein: using the identification of $A$ with the homogeneous coordinate ring of the Segre variety $\op{Seg} = \left(\bP^n\right)^{\times d}$ (where $\bP^n$ denotes $\op{Proj}(\Sym^{\bullet}(U))$) with embedding line bundle $\mc{L}=\mc{O}(1,\dots,1)$, then we have a canonical identification of the dualizing module of $A$ as
\[
\begin{aligned} 
\omega_A &= \bigoplus_{r\in\bZ} \rH^0\left(\op{Seg},\omega_{\op{Seg}}\oo\mc{L}^r\right) \\
&= \rH^0\left(\op{Seg},\omega_{\op{Seg}}\oo\mc{L}^{n+1}\right) \oo A \\
& = \rH^{dn}\left(\op{Seg},\mc{L}^{-n-1}\right)^{*} \oo A \\
& = \left(\rH^n\left(\bP,\mc{O}_{\bP}(-n-1)\right)^{\oo d}\right)^{*} \oo A,
\end{aligned}
\]
where the last equality comes from the K\"unneth formula. Since the cup product in cohomology is graded commutative (that is, $y\cup x = (-1)^{|x|\cdot|y|}x\cup y$), it follows that for the above identification, $\mf{S}_d$ acts trivially on $\rH^n\left(\bP,\mc{O}_{\bP}(-n-1)\right)^{\oo d}$ when $n$ is even, and it acts via the sign representation when $n$ is odd. If we write $\det(W) = \bigwedge^N W$ for the \defi{determinant} of an $N$-dimensional representation $W$, then we have $\rH^n\left(\bP,\mc{O}_{\bP}(-n-1)\right) = \det(U^*)$. We therefore get an $\mf{S}_d\times\GL(U)$-equivariant identification
\begin{equation}\label{eq:omegaA-as-Sd-rep}
 \omega_A = \begin{cases}
 \det(U)^{\oo d} \oo A & \text{for }n\text{ even}, \\
 \det(U)^{\oo d} \oo V_{(1^d)} \oo A & \text{for }n\text{ odd}.
\end{cases}
\end{equation}

The dualizing module for $R$ is canonically identified as
\[ \omega_R = \det(\Sym^d U) \oo R,\]
and since $R$ is a regular ring, we have
\begin{equation}\label{eq:omegaA-and-Ext-c}
 \omega_A = \Ext^c_R(A,\omega_R),\text{ where }c = \dim(R)-\dim(A) = {d+n\choose n}-dn-1.
\end{equation}
More generally, we have that $\Ext^c_R(-,\omega_R)$ defines an auto-equivalence (duality) on the category of Cohen--Macaulay $R$-modules of codimension $c$, and we write
\[ M^{\vee} = \Ext^c_R(M,\omega_R).\]
We get the following description of duality for the modules of covariants.

\begin{proposition}\label{prop:M-lam-dual}
 If $n$ is even then $M_{\lambda}^{\vee}\simeq M_{\lambda}$ as $R$-modules (we say that $M_{\lambda}$ is self-dual). If $n$ is odd then $M_{\lambda}^{\vee} \simeq M_{\lambda^{\dagger}}$, where $\lambda^{\dagger}$ denotes the transpose partition to $\lambda$.
\end{proposition}

\addtocounter{equation}{-1}
\begin{subequations}
\begin{proof}
  If we apply the duality functor to \eqref{eq:SW-decomp-A} (keeping track of the $\mf{S}_d$-action) we get
\begin{equation}\label{eq:SW-decomp-omega-A}
\bigoplus_{|\lambda|=d} M_{\lambda}^{\vee} \oo V_{\lambda} = \omega_A.
\end{equation}
When $n$ is even, we have $\omega_A\simeq A$ as $\mf{S}_d$-equivariant $R$-modules, hence $M_{\lambda}^{\vee}\simeq M_{\lambda}$ for all $\lambda$ by comparing \eqref{eq:SW-decomp-A} with \eqref{eq:SW-decomp-omega-A}. When $n$ is odd, we have $\omega_A\simeq V_{(1^d)}\oo A$ as $\mf{S}_d$-equivariant $R$-modules, hence $M_{\lambda}^{\vee} \simeq M_{\lambda^{\dagger}}$ follows from \eqref{eq:SW-decomp-A}, \eqref{eq:SW-decomp-omega-A}, and the isomorphisms $V_{\lambda} \simeq V_{\lambda^{\dagger}} \oo V_{(1^d)}$.
\end{proof}
\end{subequations}

By specializing the result above to the case $\lambda=(d)$ we get the following.

\begin{corollary}\label{cor:B-Gorenstein}
 If $n$ is even then $B$ is a Gorenstein algebra.
\end{corollary}

As a partial answer to Problem~\ref{prob:res-M-lam} we consider the shape of the minimal free resolution of $M_{\lambda}$: we noted the formula for the projective dimension in \eqref{eq:pdim-M-lam}, and our next goal is to give a bound on the \defi{Castelnuovo--Mumford regularity} of $M_{\lambda}$. Recall that for a graded $R$-module $M$ we have
\[ 
\begin{aligned}
\reg(M) &= \max\{r : \Tor_i^R(M,\bk)_{i+r}\neq 0\text{ for some }i\} \\
&= \max\{r : \Ext^j_R(M,\bk)_{-j-r}\neq 0\text{ for some }j\}
\end{aligned}
\]
When $M$ is Cohen--Macaulay of codimension $c$, it suffices to consider $i=j=c$ in the above formula. In particular, we have that $\reg(M)=r$ if and only if
\begin{equation}\label{eq:Ext-descr-reg}
 \Ext^c_R(M,\bk)_{-c-r}\neq 0\text{ and }\Ext^c_R(M,\bk)_j = 0\text{ for }j<-c-r.
\end{equation}
We show the following.

\begin{theorem}\label{thm:reg-A}
 We have that
 \[ \op{reg}(A) = (d-1)n,\]
 and therefore $\op{reg}(M_{\lambda})\leq (d-1)n$ for all $\lambda$. If $n$ is even then $\op{reg}(B)=\op{reg}(A)$, while for $n$ odd we have $\op{reg}(B)<\op{reg}(A)$.
\end{theorem}

\begin{proof}
  Let $N = \binom{n+d}{d}$. Using the earlier descriptions of $\omega_R$, $\omega_A$, we get that as graded modules
\[ \omega_R \simeq R\left(-N\right)\text{ and }\omega_A \simeq A(-n-1).\]
Using the notation \eqref{eq:omegaA-and-Ext-c} we get
\[ 
\Ext^c_R(A,R) = \Ext^c_R(A,\omega_R)\left(N \right) = \omega_A\left(N\right) = A\left(N-n-1\right).
\]
It follows from \eqref{eq:Ext-descr-reg} that $r=\reg(A)$ satisfies
\[
  -c-r = -N+n+1,
\]
which simplifies to $c=(d-1)n$, as desired. The last statement follows from the isomorphism
\[ \Ext^c_R(B,R) = \Ext^c_R(A,R)^{\mf{S}_d},\]
and \eqref{eq:omegaA-as-Sd-rep}, which shows that $\mf{S}_d$ acts trivially on the generator of minimal degree of $\omega_A$ when $n$ is even, and it acts via the sign representation when $n$ is odd. It follows that $\Ext^c_R(B,R)$ and $\Ext^c_R(A,R)$ coincide in lowest degree precisely when $n$ is even, concluding the proof.
\end{proof}

\subsection{The case $d=2$}\label{subsec:d=2} When $d=2$, we have as in Example~\ref{ex:d=2-n=2} an identification between $Y_{2,n}$ and the variety of $(n+1)\times(n+1)$ symmetric matrices of rank $\leq 2$. It follows from \cite[Theorem~6.3.1]{weyman} that $Y_{2,n}$ is normal, hence $B=\bk[Y_{2,n}]= R/J_{2,n}$, and moreover we know the minimal resolution of $B$ as an $R$-module. In particular, \cite[Corollary~6.3.7]{weyman} shows that $B$ is Gorenstein if and only if $n$ is even, so the conclusion of Corollary~\ref{cor:B-Gorenstein} is optimal in this case.

For the remaining module of covariants $M_{(1,1)}$ we have a decomposition
\[ M_{(1,1)} = \bigoplus_{m\geq 0} \bigwedge^2(\Sym^m U) = \bigoplus_{a\geq b\geq 0}\bS_{(2a+1,2b+1)}U,\]
which is the module denoted $M^s(1^2)$ in \cite[Example~6.6.11]{weyman}. Note that by Proposition~\ref{prop:M-lam-dual}, $M_{(1,1)}=\omega_B$ is the canonical module of $B$ when $n$ is odd, while for $n$ even $M_{(1,1)}$ is self-dual (see also \cite[Corollary~5.1.5]{weyman}).

\subsection{The case $n=1$}

We next assume that $\dim U = 2 = n+1$, in which case $Y_{d,1}=X_{d,1}$, since every binary form decomposes as a product of linear factors. This shows that $R=B$, and if we apply \eqref{eq:pdim-M-lam} then it follows that each $M_{\lambda}$ is a free $B$-module. Our next goal is to give an explicit equivariant decomposition of each of the modules $M_\lambda$, or equivalently, for the space of minimal generators of $M_{\lambda}$ as a $B$-module. To do so, we first introduce some combinatorial notation.

Every polynomial representation $L$ of $\GL(U)$ of degree $k$ decomposes as a direct sum of eigenspaces relative to the action of the maximal torus $(\bk^*)^2$:
\[ L = \bigoplus_{i+j=k} L_{i,j},\text{ where }t\cdot \ell = t_1^it_2^j\ell\text{ for }t=(t_1,t_2)\in(\bk^*)^2\text{ and }\ell\in L_{i,j}.\]
We define the \defi{character} of $L$ to be
\[ \op{ch}(L) = \sum_{i=0}^k \dim(L_{i,k-i}) \cdot q^i \in \bZ[q]\]
and note that the degree of $L$ together with its character completely determines $L$ as a $\GL(U)$-representation. We have for instance
\[ \op{ch}(\Sym^m U) = 1+q+\cdots+q^m =: [m+1]_q,\]
which is called a \defi{$q$-number}. For a finitely generated graded $\GL(U)$-equivariant $R$-module~$M$, we define its \defi{equivariant Hilbert series} to be
\[ \rH_M(t) = \sum_{m\in\bZ} \op{ch}(M_m) \cdot t^m.\]
For instance, the equivariant Hilbert series of $A$ and $B$ are
\begin{equation}\label{eq:Hilb-AB}
  \rH_A(t) = \sum_{m \ge 0} \left([m+1]_q\right)^d \cdot t^m, \qquad  \rH_B(t) = \sum_{m \ge 0} h_d(1,q,\dots,q^m)\cdot t^m,
\end{equation}
where  $h_d(x_0,\dots,x_m)$ is the \defi{$d$-th complete symmetric polynomial} (the sum of all degree $d$ monomials in the $x_i$). We define \defi{$q$-factorials $[m]_q!$} and \defi{$q$-binomial coefficients $\begin{bmatrix} m\\k \end{bmatrix}_q$} by 
\[[m]_q!=[m]_q[m-1]_q \cdots [1]_q \qquad \begin{bmatrix} m\\k \end{bmatrix}_q = \frac{[m]_q!}{[k]_q![m-k]_q!}.\]
It follows from \cite[Section~I.3,~Example~1]{macdonald} that
\[
  h_d(1,q,\dots,q^m) = \begin{bmatrix}m+d\\ d\end{bmatrix}_q,
\]
which combined with \cite[Section~I.2,~Example~3]{macdonald} and \eqref{eq:Hilb-AB} shows that
\[
  \rH_B(t) = \frac{1}{(1-t)(1-qt) \cdots (1-q^d t)}.
\]
To describe $\rH_A(t)$ as a rational function, we need more notation: given a permutation $\sigma \in \fS_d$, $\sigma$ has a {\bf descent} at $i$ if $\sigma(i) > \sigma(i+1)$. We let $\des(\sigma)$ denote the number of descents, and we let $\maj(\sigma)$ (the {\bf major index}) be the sum of the descents. It follows from \cite[Theorem~1]{carlitz} that
\begin{equation}\label{eq:Carlitz-HA}
  \rH_A(t) = \frac{\sum_{\sigma \in \fS_d} t^{\des(\sigma)} q^{\maj(\sigma)}}{(1-t)(1-qt) \cdots (1-q^d t)}.
\end{equation}
Since $A$ is free over $B$, the quotient
\[
  \frac{\rH_A(t)}{\rH_B(t)} = \sum_{\sigma \in \fS_d} t^{\des(\sigma)} q^{\maj(\sigma)}
\]
is the Hilbert series for the minimal generators for $A$ as a $B$-module. Notice that the maximal number of descents for a permutation $\sigma\in\mf{S}_d$ is $(d-1)$, which is compatible with the conclusion of Theorem~\ref{thm:reg-A}.

\begin{example}\label{ex:d=3-gensA}
For $d=3$, the invariants $\des(\sigma),\maj(\sigma)$ for $\sigma\in\mf{S}_3$ are as follows:
\[
\begin{array}{c|cccccc}
\sigma & 123 & 213 & 312 & 132 & 231 & 321 \\\hline
\des(\sigma) & 0 & 1 & 1 & 1 & 1 & 2 \\ \hline
\maj(\sigma) & 0 & 1 & 1 & 2 & 2 & 3 \\
\end{array}
\]
We get that
\[
\frac{\rH_A(t)}{\rH_B(t)} = 1 + t(2q+2q^2) + t^2 q^3.
\]
Since the generators of $A$ in degree $i$ are polynomial $\GL(U)$-representations of degree $3i$, we deduce:
\[\Tor_0^B(A,\bk)_0 \simeq \bk,\quad \Tor_0^B(A,\bk)_1 \simeq (\bS_{(2,1)}U)^{\oplus 2},\quad \Tor_0^B(A,\bk)_2 \simeq \bS_{(3,3)}U,\]
and in fact the groups above are the generators of $M_{(3)}$, $M_{(2,1)}^{\oplus 2}$ and $M_{(1,1,1)}$ respectively.
\end{example}

To describe the generators of each $M_{\lambda}$ we need a refinement of \eqref{eq:Carlitz-HA}. It will be useful to picture each partition $\lambda$ by its \defi{Young diagram}, consisting of left justified rows of boxes, with $\lambda_i$ boxes in row $i$. For instance, $\lambda=(5,2,1)$ has Young diagram
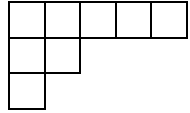
\[\yng(5,2,1)\]
For $|\lambda|=d$, a \defi{standard Young tableau} $T$ of shape $\lambda$ is a filling of the Young diagram of $\lambda$ with the numbers $1,\dots,d$ (each appearing once), which is increasing along both rows and columns. An example for $d=8$ and $\lambda=(5,2,1)$ is the tableau
\begin{equation}\label{eq:ex-T-standard}
\Yvcentermath1\young(13578,26,4)
\end{equation}
We let ${\rm SYT}(\lambda)$ denote the set of standard Young tableaux of shape $\lambda$, and recall that they can be used to index a basis of irreducible $\mf{S}_d$-representation $V_{\lambda}$. We set
\[ f^{\lambda} := |{\rm SYT}(\lambda)| = \dim(V_{\lambda}),\]
and note that \eqref{eq:SW-decomp-A} implies
\[A = \bigoplus_{|\lambda|=d} M_\lambda^{\oplus f_\lambda}.\]

Given a standard Young tableau $T$, we say $T$ has a {\bf descent} at $i$ if $i+1$ appears in a lower row than $i$. We define $\des(T)$ to be the number of descents of $T$ and $\maj(T)$ to be the sum of descents of $T$. For $T$ as in \eqref{eq:ex-T-standard}, we have descents at $1,3,5$, hence
\[ \des(T) = 3\text{ and }\maj(T) = 9.\]
The RSK algorithm \cite[\S 7.11]{EC2} gives a bijection $\sigma \mapsto (P(\sigma),Q(\sigma))$ between $\fS_d$ and pairs of standard Young tableaux of the same shape (and of size $d$), with the property that $\sigma$ and $Q(\sigma)$ have the same set of descents \cite[Lemma~7.23.1]{EC2}. In particular, this implies that
\[
  \sum_{\sigma \in \fS_d} t^{\des(\sigma)} q^{\maj(\sigma)} = \sum_{|\lambda|=d} f^\lambda \sum_{T \in {\rm SYT}(\lambda)} t^{\des(T)} q^{\maj(T)}.
\]
We will show that the decomposition above is compatible with \eqref{eq:SW-decomp-A}, reflecting the distribution of the generators of $A$ among the generators of the summands $M_{\lambda}$. In the next result, we let $s_\lambda(x_1,\dots,x_k)$ denote the Schur polynomial indexed by $\lambda$ \cite[\S 7.10]{EC2}, which is the character of the Schur functor $\bS_\lambda(\bC^k)$.

\begin{proposition}
  Let $\lambda$ be a partition of $d$. We have
  \[
    \rH_{M_\lambda}(t) = \sum_{m \ge 0} s_\lambda(1,q,\dots,q^m) t^m = \frac{\sum_{T \in {\rm SYT}(\lambda)} t^{\des(T)} q^{\maj(T)}}{(1-t)(1-qt)\cdots (1-q^d t)}.
  \]
\end{proposition}

\begin{proof}
  The first equality follows from \eqref{eq:GL-decomp-M-lam}, so we only need to verify the second one. By \cite[Proposition 7.19.12]{EC2}, we have
  \[
    s_\lambda(1,q,\dots,q^m) = \sum_{T \in {\rm SYT}(\lambda)} \begin{bmatrix} m-\des(T)+d \\ d \end{bmatrix}_q q^{\maj(T)}.
  \]
  If we sum over all $m \ge 0$ then we get
  \begin{align*}
    \sum_{m \ge 0} s_\lambda(1,q,\dots,q^m)t^m
    &= \sum_{T \in {\rm SYT}(\lambda)} q^{\maj(T)} \sum_{m \ge 0} \begin{bmatrix} m-\des(T)+d\\d\end{bmatrix}_q t^m\\
    &= \sum_{T \in {\rm SYT}(\lambda)} q^{\maj(T)} \sum_{m \ge 0} \begin{bmatrix} m+d\\d\end{bmatrix}_q t^{m+\des(T)}\\
    &= \sum_{T \in {\rm SYT}(\lambda)} q^{\maj(T)}\frac{t^{\des(T)}}{(1-t)(1-qt)\cdots (1-q^dt)},
  \end{align*}
  which proves the result.
\end{proof}

As a consequence, we get the following description of the minimal generators of $M_{\lambda}$.

\begin{corollary}\label{cor:mingens-Mlam}
  As a $B$-module, $M_\lambda$ has one generator for every standard Young tableau $T$ of shape $\lambda$, and the generator corresponding to $T$ lies in degree $\des(T)$. Furthermore, 
  \[ \op{ch}\left(\Tor_0^B(M_{\lambda},\bk)_i\right) = \sum_{\substack{T \in {\rm SYT}(\lambda)\\ \des(T)=i}} q^{\maj(T)}.\]
\end{corollary}

Expressions such as the one in Corollary~\ref{cor:mingens-Mlam} have been widely studied in combinatorics (see for instance \cites{KR,gos,CEKS,keith,GZ}), but this is the first time we encountered them in an invariant theoretic setting.

\section{Equations of Chow variety in characteristic zero} \label{s:char-zero}

In this section we continue to assume that $\bk$ is algebraically closed of characteristic zero. We will briefly discuss some generalities regarding equations for Chow varieties, and then proceed with a number of examples.

\subsection{Brill's equations}

Brill's equations were introduced originally in \cite{brill1,brill2} (see also the exposition by Gordan \cite{gordan}) and then described in \cite[\S 4.2]{gkz}. These are equations of degree $d+1$ which define $Y_{d,n}$ set-theoretically and span a subrepresentation of $\Sym^d U \otimes \Sym^d U \otimes \Sym^{d(d-1)} U$. In fact, one has a more precise result of Guan \cite{guan}.

\begin{proposition}[Guan] \label{prop:guan}
  Suppose $\dim U\ge 3$ and $d\ge 2$. The $\GL(U)$-module given by the span of Brill's equations has the following decomposition into irreducible representations:
\begin{enumerate}[\rm \indent (a)]
\item $\bS_{(7,3,2)} U$ for $d=3$.
\item $\bigoplus_{j=2}^d \bS_{(d^2-j,d,j)}U$ for $d\ne 3$.
\end{enumerate}
\end{proposition}

Let $L_{d,n}$ denote the ideal generated by Brill's equations. It is an interesting question to compare the ideals $L_{d,n}$ and $J_{d,n}$. For a general pair $(d,n)$ these ideals are not equal. To see this, 
let $D_{d,n,r}$ denote the subvariety of symmetric tensors of degree $d$ in $n+1$ variables, which have subspace rank $\le r$, that is, after a change of basis in $U$, they can be rewritten using at most $r$ variables. There is a natural set of equations vanishing on $D_{d,n,r}$ consisting of $(r+1)\times (r+1)$ minors of a matrix of the form
\begin{align*} 
  \psi \colon U\otimes R_{d,n}(-1)\rightarrow \Sym^{d-1}U^*\otimes R_{d,n}.
\end{align*}

We let $I_{d,n,r+1}$ denote the ideal generated by the $(r+1)\times (r+1)$ minors of $\psi$.
This ideal is non-zero for $1\le r\le n$: see \cite[\S 7.2]{weyman}, particularly Corollary~7.2.3, for more information on such ideals. The varieties $D_{d,n,r}$ and their defining equations were analyzed in \cite{porras}. It is clear that $Y_{d,n}\subset D_{d,n,d}$, since the expression $f = \ell_1\ell_2\cdots\ell_d$ shows that $f$ can be written using at most $d$ variables up to a change of basis.

\begin{proposition}\label{prop:notequal}
  For $n \ge d \ge 3$, the ideals $L_{d,n}$ and $J_{d,n}$ are different.
\end{proposition}

\begin{proof}
  Since $Y_{d,n}\subset D_{d,n,d}$, we have $I_{d,n,d+1} \subseteq J_{d,n}$. By \cite[\S 7.2]{weyman}, the minors generating $I_{d,n,d+1}$ span a subrepresentation of $\bigwedge^{d+1}U\otimes\bigwedge^{d+1}(\Sym^{d-1}U)$, which is non-zero since $n\geq d$. By Pieri's formula \cite[Corollary 2.3.5]{weyman}, if $\bS_\lambda U$ appears in $\Sym^d U\otimes \Sym^d U\otimes \Sym^{d(d-1)} U$, then $\ell(\lambda) \le 3$. However, if $\bS_\lambda U$ appears in $\bigwedge^{d+1} U \otimes \bigwedge^{d+1}(\Sym^{d-1} U)$, then $\ell(\lambda) \ge d+1 > 3$, hence $I_{d,n,d+1} \not\subseteq L_{d,n}$ (since both ideals are generated in degree $d+1$).
\end{proof}

As noted in Section~\ref{subsec:d=2}, when $d=2$ we have that $Y_{2,n}$ is the space of rank $\le 2$ symmetric matrices, hence the ideal $J_{2,n}$ is generated by the $3 \times 3$ minors of the symmetric matrix. The first interesting case is therefore $d=3$, which we consider next.

\subsection{The case $d=3$}

\begin{proposition}\label{prop:d=3}
If $d=3$ then $L_{3,n}+I_{3,n,3} = J_{3,n}$.
\end{proposition}

\begin{proof}
  It suffices to check that $L_{3,n}$ and $J_{3,n}$ are equal modulo $I_{3,n,3}$. To do that, we use the fact (see \cite[Corollary 7.2.3]{weyman}) that $I_{3,n,3}$ consists of all subrepresentations of $R_{3,n}$ that are isomorphic to $\bS_\lambda U$ with $\ell(\lambda)\ge 4$. Hence every Schur functor $\bS_\lambda U$ that appears in $R_{3,n}/I_{3,n,3}$ has $\ell(\lambda) \le 3$, so that equality can be checked when $\dim U = 3$. This case follows from Corollary~\ref{cor:d3n3} below.
\end{proof}

\begin{remark}
  For general $d$, the argument in the previous proof tells us that to understand $J_{d,n}$, we can work modulo $I_{d,n,d}$, and then it suffices to consider the case $n=d$.
\end{remark}

\subsubsection { Case of $d=3$, $n=2$}
We now deal with the case $d=3$, $n=2$. We denote $R:=R_{3,2}=\Sym (\Sym^3 U)$, with $\dim U=3$.
We write the generic cubic form as
\[
  f(x,y,z)=\sum_{\alpha ,\beta ,\gamma} a_{\alpha\beta \gamma}x^\alpha y^\beta z^\gamma.
\]
We can compute the minimal free resolution of the $R$-module $B_{3,2}$.

\begin{example}\label{thm:a_{33}}
  The equivariant minimal free resolution of the algebra $B_{3,2}$ considered as an $R$-module has the form
\[
  0\rightarrow {\begin{array}{c} \bS_{5,5,5}U\otimes R(-5)\\\oplus\\ \bS_{7,7,7}U\otimes R(-7) \end{array}} \xrightarrow{d_3} \bS_{5,5,2}U\otimes R(-4) \xrightarrow{d_2} \bS_{5,2,2}U\otimes R(-3) \xrightarrow{d_1}  {\begin{array}{c} R \\ \oplus \\ \bS_{2,2,2} U \otimes R(-2) \end{array}}.
\]
The map $d_2$ is defined uniquely up to a non-zero scalar by the equivariance condition. The map $d_1$ has the linear component that is just a vector of coefficients $a_{\alpha\beta\gamma}$ and the cubic component which consists of coefficients of the Hessian covariant 
\[
  H(f(x,y,z))=\det \left(\begin{matrix} \partial^2f\over{\partial}x^2&\partial^2 f\over{\partial x\partial y}&\partial^2 f\over{\partial x\partial z}\\
\partial^2f\over{\partial x\partial y}&\partial^2 f\over{\partial y^2}&\partial^2 f\over{\partial y\partial z}\\
\partial^2f\over{\partial x\partial z}&\partial^2 f\over{\partial y\partial z}&\partial^2 f\over{\partial z^2}
\end{matrix}\right).
\]
There are two non-zero scalars involved, but they both need to be non-zero, so up to a change of basis in $\bS_{2,2,2}U\otimes R(-2)\oplus R$ the choice is unique.

The easiest way to deal with this is to exhibit the matrix of $d_2$ explicitly. We identify $\bS_{5,2,2}U$ with $(\det U)^{\otimes 2}\otimes \Sym^3 U$ and $\bS_{5,5,2}U$ with $(\det U)^{\otimes 5}\otimes \Sym^3U^*$. we get the following skew-symmetric matrix for $d_2$:
\[
  \begin{bmatrix}0&0&0&a_{003}&0&0&-3a_{012}&0&3a_{021}&-a_{030}\\
0&0&-3a_{003}&0&0&6a_{012}&3a_{102}&-3a_{021}&-6a_{111}&3a_{120}\\
0&3a_{003}&0&0&-3a_{012}&-6a_{102}&0&6a_{111}&3a_{201}&-3a_{210}\\
-a_{003}&0&0&0&3a_{102}&0&0&-3a_{201}&0&a_{300}\\
0&0&3a_{012}&-3a_{102}&0&-6a_{021}&6a_{111}&3a_{030}&-3a_{120}&0\\
0&-6a_{012}&6a_{102}&0&6a_{021}&0&-6a_{201}&-6a_{120}&6a_{210}&0\\
3a_{012}&-3a_{102}&0&0&-6a_{111}&6a_{201}&0&3a_{210}&-3a_{300}&0\\
0&3a_{021}&-6a_{111}& 3a_{201}&-3a_{030}&6a_{120}&-3a_{210}&0&0&0\\
-3a_{021}&6a_{111}&-3a_{201}&0&3a_{120}&-6a_{210}&3a_{300}&0&0&0\\
a_{030}&-3a_{120}&3a_{210}&-a_{300}&0&0&0&0&0&0
\end{bmatrix}.
\]
We also know the form of  $d_1$ and the kernel of the transpose which is $d_3$ is its transpose since the matrix of $d_2$ is skew-symmetric. 

To see this resolution resolves the algebra $B_{3,3}$ note that the square of the extra generator $\bS_{2,2,2}U$ is in the image of the $R$-submodule generated by the unit generator $\bS_{0,0,0}U$,
so it can be viewed as the square root of the Aronhold invariant $\bS_{4,4,4}U\subset \Sym^4(\Sym^3U)$. See \cite{ottaviani} for the explicit description of an Aronhold invariant. Note that in order to prove it is non-zero, he evaluates it on the polynomial $xyz$, which shows that this invariant is not in the ideal $J_{3,3}$.

Finally, one can prove acyclicity by hand using the Buchsbaum--Eisenbud acyclicity criterion. More precisely,
the rank conditions are obvious. The $2\times 2$ minors of $d_1$ generate the ideal of depth $3$ as set-theoretically they give the Chow variety.
Indeed, one needs to check that the form $f(x,y,z)$ is proportional to its Hessian $H(f(x,y,z))$ precisely when $f(x,y,z)$ is in the Chow variety.
This fact is attributed to Aronhold \cite{aronhold}.
One way to see it geometrically is to observe that the common zeros of $f(x,y,z)$ and its Hessian are inflection points of the curve $C$ given by $f(x,y,z)$ so if $f(x,y,z)$ and $H(f(x,y,z))$ are proportional,
then every point of $C$ is its inflection point, so $C$ is a union of lines.

It remains to check that maximal nonvanishing minors of $d_2$ generate an ideal of depth $2$. But they also give set-theoretically the Chow variety. This can also be checked directly on orbit representatives, as orbits of ternary cubics are known.
\end{example}

We continue with drawing consequences from Theorem~\ref{thm:a_{33}}.

\begin{corollary} \label{cor:d3n3}
\begin{enumerate}[\rm \indent (a)]
\item The module $C_{3,2}$ is isomorphic to $(\det U)^{\otimes 2}(-2)$.
\item The ideal $J_{3,2}$ is generated by the Brill equations, i.e., the representation $\bS_{(7,3,2)}U$ in degree $4$.
\end{enumerate}
\end{corollary}

\begin{proof}
  The first statement follows from the fact that the third graded component of $\phi_{3,2}$ is an isomorphism. This can be done by calculating the value of $\phi_{3,2}$ on highest weight vectors. Indeed, it is well known that (or by using software such as LiE \cite{LiE})
  \[
    \Sym^3(\Sym^3U)=\bS_{9}U\oplus \bS_{7,2}U\oplus \bS_{6,3}U\oplus \bS_{5,2,2}U\oplus \bS_{4,4,1}U.
  \]
  Since for $\dim(U)=2$ the analogous map $\phi_{3,1}$  is an isomorphism, we just need to show that $\phi_{3,2}$ applied to the highest weight vectors of $\bS_{5,2,2}U$ and of $\bS_{4,4,1}U$
  are not zero. We already mentioned that $\bS_{5,2,2} U$ is a Hessian covariant. The highest weight vector of $\bS_{4,4,1}U$ is the determinant
  \[
    \det\left(\begin{matrix}6a_{300}&2a_{210}&2a_{120}\\
  2a_{210}&2a_{120}&6a_{030}\\
  2a_{201}&a_{111}&2a_{021}\end{matrix}\right).
\]
  Using these formulas we can prove the result.
To prove the second statement we notice that the resolution of $C_{3,2}$ is the Koszul complex tensored with $(\det U)^{\otimes 2}$.
The mapping cone of the map of complexes lifting the map $B_{3,2}\rightarrow C_{3,2}$ is a non-minimal free resolution of the $R$-module $R/J_{3,2}$. This proves that $J_{3,2}$ is generated in degrees $\le 4$, more precisely by a subrepresentation of $\bS_{(5,5,2)}U\oplus \bS_{(7,3,2)}U$. Then it is enough to see that $\bS_{(5,5,2)}U$ (which occurs with multiplicity 1 in $\Sym^3(\Sym^3 U)$) does not vanish on $Y_{3,2}$. This is clear as when evaluating Hessian on the product $xyz$ we get a non-zero polynomial $2xyz$.
\end{proof}

\subsubsection{The case $d=n=3$}

Now we consider the case $\dim (U)=4$. 
We prove the following result, even stronger than Proposition~\ref{prop:notequal}.

\begin{corollary}
  In the case $d=n=3$ the ideals $J_{3,3}$ and $L_{3,3}$ do not define the same subscheme of $\operatorname{Proj} (R_{3,3})$, i.e., the saturations of $J_{3,3}$ and $L_{3,3}$ are not the same.
\end{corollary}

\begin{proof}
  As discussed above, the ideal $I_{3,3,3}$ is generated in degree $4$ and the linear span of its generators  equals $\bigwedge^4 U \otimes\bigwedge^4 (\Sym^2U)\subset \Sym^4 (\Sym^3U)$. Since $\bigwedge^4 (\Sym^2U)$ contains $\bS_{(5,1,1,1)}U$, we can find a representation $\bS_{(6,2,2,2)}U$ in degree $4$ in $I_{3,3,3}$, and hence in $J_{3,3}$. We claim that this representation does not belong to the saturation of $L_{3,4}$. By Proposition~\ref{prop:guan}, the ideal $L_{3,4}$ is generated by $\bS_{(7,3,2)}U$.

Let $v$ be a highest weight vector in $\bS_{(6,2,2,2)}U$ and let $w$ be a highest weight vector in $\Sym^3 U$. If $\bS_{(6,2,2,2)}U$ is in the saturation of $L_{3,4}$, then $vw^m \in L_{3,4}$ for $m \gg 0$. But this is a highest weight vector for $\bS_{(6+3m,2,2,2)}U$. Since $(6+3m,2,2,2)$ does not contain the partition $(7,3,2)$, it cannot appear in any tensor product of the form $\bS_{(7,3,2)} U \otimes \bS_\lambda U$ as a consequence of the Littlewood--Richardson rule \cite[Theorem 2.3.4]{weyman}, which proves the claim.
\end{proof}

\subsection{The case $d=4$, $n=2$}

Finally we collect together some known results about the smallest case of degree $d=4$. Let $\dim U=3$.

\begin{proposition} 
\begin{enumerate}[\rm\indent (a)]
\item The $4$th graded component of the map $\phi_{4,2}$
  \[
    \phi_{4,2,4} \colon \Sym^4 (\Sym^4 U)\to \Sym^4( \Sym^4 U)
  \]
is an isomorphism.
\item The module $C_{4,2}$ has only two graded components:
  \[
    (C_{4,2})_2 = \bS_{(4,2,2)}U, \qquad (C_{4,2})_3= \bS_{(7,3,2)}U.
  \]
\item The Brill  ideal $L_{4,2}$ is generated by the representations
  \[
    \bS_{(14,4,2)}U \oplus \bS_{(13,4,3)}U\oplus \bS_{(12,4,4)}U.
  \]
\item The Brill  ideal $L_{4,2}$ is not radical. 
\end{enumerate}
 \end{proposition}

 \begin{proof}
   The first part follows from the computational verification from \cite{Muller}. The second part involves only calculations in degrees $2$ and $3$.
In degree $2$ the calculation is clear since all partitions in $\Sym^2(\Sym^4U)$ have at most two parts so they already appear in the $\dim(U)=2$ case.
The calculation in degree $3$ proceeds as follows. We calculate both domain and codomain using the computer program LiE \cite{LiE}.
We see that the only questionable representations are $\bS_{(8,2,2)}U$,  $\bS_{(7,4,1)}U$ and  $\bS_{(6,4,2)}U$.
We will describe the highest weight vectors corresponding to these representations.
Let us denote our form as
$$f(x,y,z)=\sum_{\alpha+\beta+\gamma=4} a_{\alpha\beta\gamma}x^\alpha y^\beta z^\gamma$$

Let us construct the highest weight vectors in our three representations.

The covariant $\bS_{(8,2,2)}U$ is the Hessian covariant $H(f(x,y,z))$ of the form $f(x,y,z)$, so the highest weight vector is just the coefficient of $x^6$ of the Hessian.
The highest weight vector from $\bS_{(7,4,1)}U$ can be constructed as follows.
The embedding of $\bS_{(7,4,1)}U$ into $\Sym^3 (\Sym^4 U)$ is a composition of the embedding of $\bS_{(6,3,0)}U$ into $\bigwedge^3 (\Sym^3 U)$ tensored with $\bigwedge^3 U$ composed with the map
\[
  \bigwedge^3U\otimes\bigwedge^3(\Sym^3U)\rightarrow \Sym^3(\Sym^4 U)
\]
embedding by $3\times 3$ minors into $\Sym^3(U\otimes \Sym^3U)$ and then multiplying $U\otimes \Sym^3U$ into $\Sym^4U$. The highest weight vector $\bS_{(6,3,0)} U$ in $\bigwedge^3 (\Sym^3 U)$ is just $e_1^3\wedge e_1^2e_2\wedge e_1e_2^2$. Combining this information we get our highest weight vector.
The highest weight vector of $\bS_{(6,4,2)}U$ is constructed similarly, by tensoring the highest weight vector of $\bS_{(4,2,0)}U$ in $\Sym^3(\Sym^2 U)$ with $\bigwedge^3 U$ twice.

 Now it is not difficult to see that the map $\phi_{4,2,4}$ takes all three highest weight vectors  to nonzero elements.
 
The third part is a special case of Proposition \ref{prop:guan}. For the last statement, see the numerics in  \cite[\S 3]{briand2}. 
\end{proof}

\subsection{The case $d=4, n=3$}
Let us analyze this case. We have the following calculations that can be done via computer programs.

\begin{proposition}
  Let $\dim U=4$. 
\begin{enumerate}[\rm \indent (1)]
\item The $4$-th graded component of the map
  \[
    \phi_{4,3,4} \colon \Sym^4 (\Sym^4 U)\to \Sym^4 (\Sym^4 U)
  \]
is surjective.
\item The module $C_{4,3}$ has only two graded components:
  \begin{align*}
    (C_{4,3})_2 &= \bS_{(4,2,2)}U\oplus \bS_{(2,2,2,2)}U,\\
    (C_{4,3})_3 &= \bS_{(7,3,2)}U\oplus \bS_{(5,4,2,1)}U\oplus \bS_{(6,2,2,2)}U.
  \end{align*}
\item The Brill ideal $L_{4,3}$ is generated by the representations
  \[
    \bS_{(14,4,2)}U\oplus \bS_{(13,4,3)}U\oplus \bS_{(12,4,4)}U.
  \]
\item The ideal $L_{4,3}+I_{4,3,4}$ is not radical.
\end{enumerate}
 \end{proposition}

\begin{proof} The first part follows from the computational verification from \cite{Muller}. The second part involves only calculations in degrees $2$ and $3$. In degree $2$ the calculation is clear since all partitions appearing in $\Sym^2(\Sym^4U)$  already appear for $\dim(U)=2$. In degree $3$ we calculate both domain and codomain by the computer program LiE \cite{LiE}. We see that in $\Sym^4(\Sym^3 U)$ we get only either partitions with $3$ parts, or the partitions that do not occur in $\Sym^3(\Sym^4U)$. So the calculation follows from the $d=4$, $n=2$ case.  The third is a special case is part of Proposition~\ref{prop:guan}. The last statement is part of Proposition~\ref{prop:d=3}. 
\end{proof}

\section{Hermite action}\label{sec:Hermite}

We return to the general situation of a vector space $U$ of arbitrary dimension $n+1$ over an arbitrary algebraically closed field $\bk$. Consider the action of $\fS_d$ on $A_{d,n}$, which recall is a finitely generated module over $B_{d,n} = \bigoplus_{m \ge 0} \rD^d(\Sym^m U)$ (which itself is a finitely generated module over $\Sym(\rD^d U)$). We have submodules
\[
  \bM(U)_{i,d-i} = \bigoplus_{m \ge 0} \bigwedge^i(\Sym^m U) \otimes (\Sym^m U)^{\otimes d-i}
\]
of $A_{d,n}$ which are again finitely generated $B_{d,n}$-modules. Multiplication by $\rD^d U$ is the composition of the following maps (where $\Delta$ is comultiplication for divided powers)
\begin{align*}
  \rD^d U \otimes \bigwedge^i(\Sym^n U) \otimes (\Sym^n U)^{\otimes d-i} &\xrightarrow{\Delta \otimes 1 \otimes 1} \rD^i U \otimes U^{\otimes d-i} \otimes \bigwedge^i(\Sym^n U) \otimes (\Sym^n U)^{\otimes d-i} \\
  &\to \bigwedge^i(U \otimes \Sym^n U) \otimes (\Sym^{n+1} U)^{\otimes d-i} \\
  &\to \bigwedge^i(\Sym^{n+1} U) \otimes (\Sym^{n+1} U)^{\otimes d-i}.
\end{align*}

In particular, if $R$ is a graded quotient ring of $\Sym(U)$, we get a finitely generated quotient module
\[
  \bM(U,R)_{i,d-i} = \bigoplus_{n \ge 0} \bigwedge^i(R_n) \otimes R_{(d-i)n}.
\]

\begin{proposition}
  The following diagram commutes, where the top and bottom left horizontal maps are the comultiplication map on the $i$-th exterior power, the middle left map is the tensor product of the comultiplication map on the $i$-th divided power tensored with the comultiplication map on the $i$-th exterior power, the right horizontal maps are multiplication in $R$, the top vertical maps are divided power comultiplication, and the bottom vertical maps are the components of the map described above:
 \[
  \xymatrix{
    \rD^d U \otimes \bigwedge^i(R_n) \otimes R_{(d-i)n} \ar[r] \ar[d] &
    {\begin{array}{c} \rD^d U \otimes \bigwedge^{i-1}(R_n) \\ \otimes R_n \otimes R_{(d-i)n}
     \end{array}}
     \ar[r] \ar[d] &    
    \rD^d U \otimes \bigwedge^{i-1}(R_n) \otimes R_{(d-i+1)n} \ar[d] \\
{\begin{array}{c}    \rD^i U \otimes U^{\otimes d-i} \otimes\\ \bigwedge^i(R_n) \otimes R_{(d-i)n}
 \end{array}}\ar[r] \ar[d] &
{\begin{array}{c}    \rD^{i-1} U \otimes U \otimes U^{\otimes d-i} \otimes\\ \bigwedge^{i-1}(R_n) \otimes R_n \otimes R_{(d-i)n}
 \end{array}}
 \ar[r] \ar[d] &
{\begin{array}{c}    \rD^{i-1} U \otimes U \otimes U^{\otimes d-i}  \otimes\\ \bigwedge^{i-1}(R_n) \otimes R_{(d-i+1)n}
 \end{array}
 }\ar[d] \\
    \bigwedge^i(R_{n+1}) \otimes R_{(d-i)(n+1)} \ar[r] &
{\begin{array}{c}    \bigwedge^{i-1}(R_{n+1}) \otimes R_{n+1}\\ \otimes R_{(d-i)(n+1)}
 \end{array}}
\ar[r] &    
    \bigwedge^{i-1}(R_{n+1}) \otimes R_{(d-i+1)(n+1)} 
    }
\]
\end{proposition}

\begin{proof}
 The top left square commutes by coassociativity of comultiplication for divided powers. The top right square commutes since the compositions are tensor products of maps that do not interact. The bottom two squares commute by definition of the action.
\end{proof}

\begin{corollary}
  Let $R$ be a graded quotient of $\Sym(U)$. Let $R[n]$ be the $n$th Veronese subring of $R$. Then for each $i,d \in \bZ$,
  \[
    \bigoplus_{n \ge 0} \Tor_i^{\Sym(R_n)}(R[n], \bk)_d
  \]
  is a finitely generated $\Sym(\rD^d U)$-module which is supported on the Chow variety. In particular, the dimension of the Tor group is eventually a polynomial in $n$ of degree $\le (\dim U-1)d$.
\end{corollary}

The polynomiality statement generalizes \cite[Theorem 4]{yang}, which was proved for coordinate rings of smooth projective varieties in characteristic zero.

\begin{proof}
  This is the homology of a Koszul complex, and the previous result shows that this Koszul complex is compatible with the $\Sym(\rD^d U)$-module structure.
\end{proof}

\begin{remark}
  There are a number of different ways that this result can be generalized without much extra effort.

  \begin{enumerate}
  \item Let $M$ be a finitely generated graded $R$-module. We can replace $R[n]$ by $M[n]$. Furthermore, if $m$ is a fixed integer, we can replace $M[n]$ by the shifted module $\bigoplus_{r \ge 0} M_{rn+m}$.
  \item We can replace $\bk$ by some other module whose resolution is linear and given by Schur functors in some fixed way. For example, for fixed $e$, we can replace it by the quotient of $\Sym(R_n)$ by the $e$th power of the maximal ideal. Then the exterior powers are replaced by hook shapes (we might need characteristic 0 for this to work).
    
  \item We can work with multigraded rings. \qedhere
  \end{enumerate}
\end{remark}

\begin{remark}
Let $\bk$ be a field of characteristic 0 and take $R = \Sym(U)$. Then the module above is a twisted commutative algebra over $\Sym(\rD^d)$ (see \cite{expos}). In fact, it's automatically finitely generated: every Schur functor $\bS_\lambda$ that appears satisfies $\ell(\lambda) \le d$, so all of the information can be detected with a vector space of dimension $d$, i.e., this twisted commutative algebra is bounded \cite[Proposition 9.1.6]{expos}.
\end{remark}

\end{document}